\documentclass[dvips,preprint]{imsart}
\usepackage{amsthm,amsmath,amssymb,bm}
\usepackage{graphicx}
\newtheorem{theorem}{Theorem}
\newtheorem{col}{Corollary}
\newtheorem{lemma}{Lemma}
\newtheorem{prop}{Proposition}

\startlocaldefs
\endlocaldefs
\allowdisplaybreaks
\begin{document}

\begin{frontmatter}

\title{Strong Consistency of Factorial $K$-means Clustering}
\runtitle{Strong Consistency of Factorial $K$-means Clustering}


\author{\fnms{Yoshikazu} \snm{Terada}\corref{}\ead[label=e1]{terada@sigmath.es.osaka-u.ac.jp}}
\address{Graduate School of Engineering Science, Osaka University, 1-3 Machikaneyama,\\ Toyonaka, Osaka, Japan\\ \printead{e1}}
\affiliation{Osaka University}

\runauthor{Y. Terada}

\begin{abstract}
Factorial $k$-means (FKM) clustering is a method for clustering objects in 
a low-dimensional subspace.
The advantage of this method is that
the partition of objects and the low-dimensional subspace reflecting the cluster structure are obtained, 
simultaneously.
In some cases that the reduced $k$-means clustering (RKM) does not work well, 
FKM clustering can discover the cluster structure underlying a lower dimensional subspace.
Conditions that ensure the almost sure convergence of 
the estimator of FKM clustering as the sample size increases unboundedly are derived.
The result is proved for a more general model including FKM clustering.
\end{abstract}


\begin{keyword}
\kwd{subspace clustering}
\kwd{$k$-means}
\end{keyword}

\end{frontmatter}

\section{Introduction}
\label{intro}

If we apply a cluster analysis to data, 
it is highly unlikely that 
all variables relate to the same cluster structure.
Hence, it is sometimes beneficial to regard the true cluster structure of interest 
as lying in a low-dimensional subspace of the data.
In these cases, 
researchers often apply the following two-step procedure:
\begin{list}{\textbullet}{\topsep=1pt}
\setlength{\parskip}{0cm}
\setlength{\itemsep}{0cm}
\item[\bf Step $1$.]
Carry out principal component analysis (PCA) and obtain the first few components. 
\item[\bf Step $2$.]
Perform the usual $k$-means clustering for the principal scores on the first few principal components, 
which are obtained in Step $1$.
\end{list}
This procedure is called ``tandem clustering" by Arabie and Hubert (1994).
Several authors warn against the use of tandem clustering 
(e.g., Arabie and Hubert (1994); Chang (1994); De Soete and Carroll (1994)).
The first few principle components of PCA do not necessarily reflect the cluster structure in data.
Thus, 
an appropriate clustering result  might not be obtained using this procedure.

Instead of a two-step procedure, such as tandem clustering, 
some methods that perform cluster analysis and dimension reduction simultaneously 
have been proposed (e.g., De Soete and Carroll (1994); Vichi and Kiers (2001)).
De Soete and Carroll (1994) proposed reduced $k$-means (RKM) clustering, 
 which includes conventional $k$-means clustering as a special case.
For given data points $\bm{x}_1,\;\dots,\;\bm{x}_n$ in $\mathbb{R}^{p}$, 
the fixed cluster number $k$ and the dimension number of subspace $q\;(q<\min\{k-1,\;p\})$,
the objective function of RKM clustering is defined by
$$
RKM_n(F,\;A) := \frac{1}{n}\sum_{i=1}^{n}\min_{1\le j \le k}\|\bm{x}_i-A\bm{f}_j\|^2,
$$
where $\bm{f}_j\in \mathbb{R}$, $F=\{\bm{f}_1,\;\dots,\;\bm{f}_{k}\}\subset \mathbb{R}^{q}$, 
$A$ is a $p\times q$ column-wise orthonormal matrix, 
and $\|\cdot\|$ represents the usual norm.
Under certain regularity conditions, 
RKM clustering has strong consistency (Terada (2012)). 
However, 
when the data matrix $X=\left( x_{ij}\right)_{n\times p}$ has a full rank, i.e., $\mathrm{rank}(X)=p$, 
RKM clustering may fail to find a subspace that reflects the cluster structure.
Indeed, RKM clustering has been applied to data composed of a total of $12$ independent variables (Figure $\ref{data}$), 
which consists of $2$ variables actually related to the cluster structure and $10$ noise variables.
\begin{figure}
\begin{minipage}{0.47\textwidth}
\begin{center}
\includegraphics[scale=0.3]{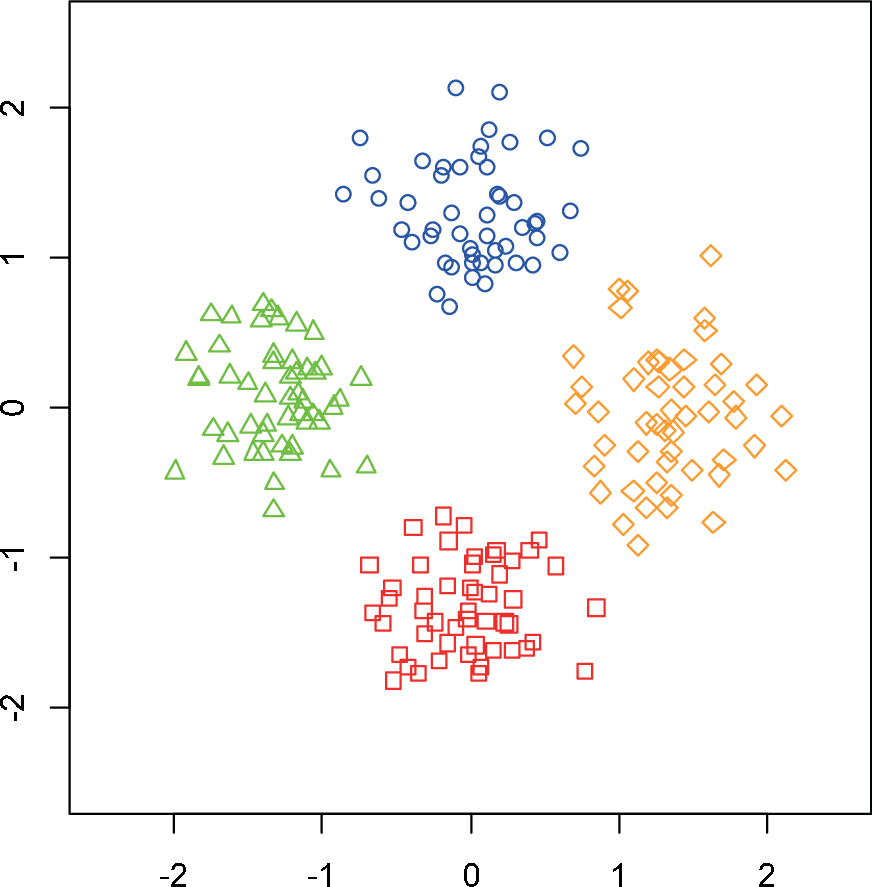}
\par{(a)}
\end{center}
\end{minipage}
\hfill
\begin{minipage}{0.47\textwidth}
\begin{center}
\includegraphics[scale=0.3]{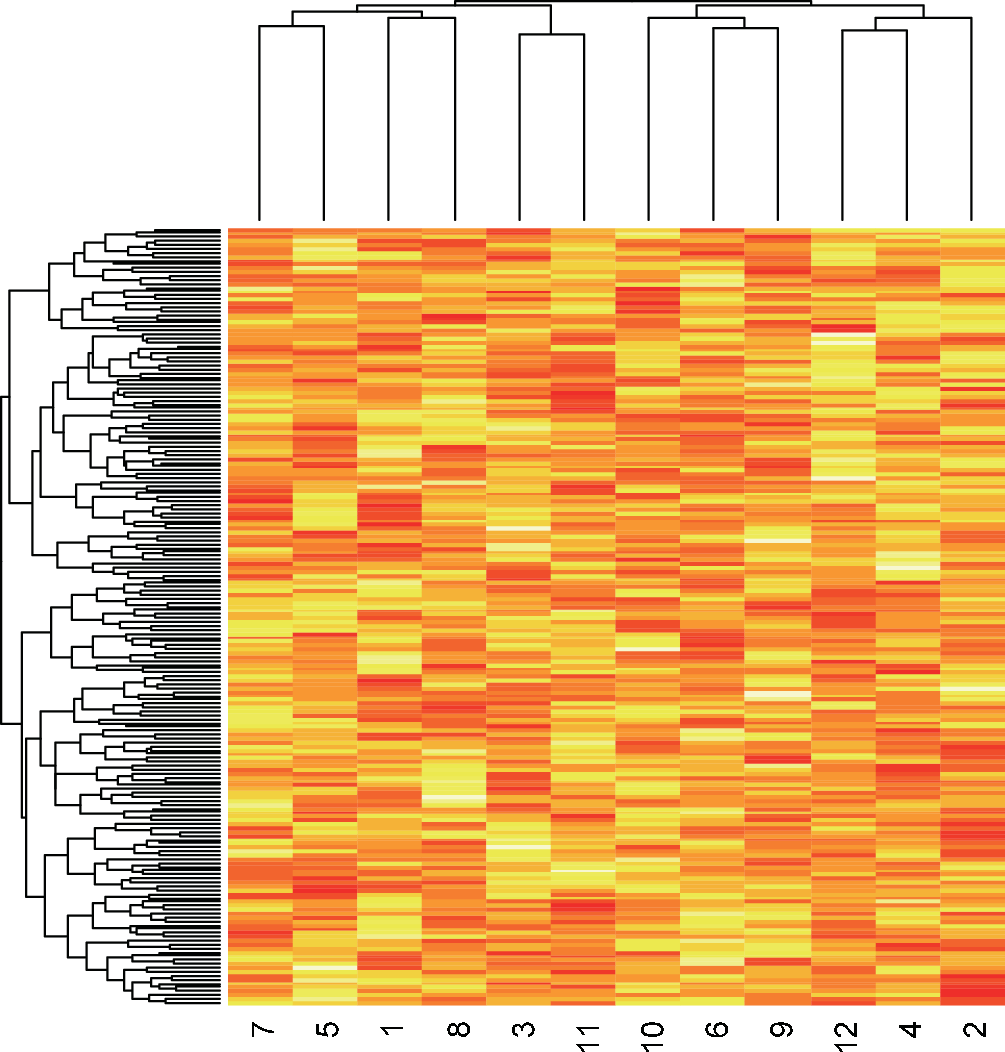}
\par{(b)}
\end{center}
\end{minipage}
\caption{Artificial data used to evaluate RKM clustering: (a) plot of two variables related to a cluster structure and (b) heat map of $12$ variables.}
\label{data}
\end{figure}
The result of RKM clustering for the data shown in Figure $\ref{data}$ is given in Figure $\ref{RKM}$.
\begin{figure}
\centerline{\includegraphics[scale=0.35]{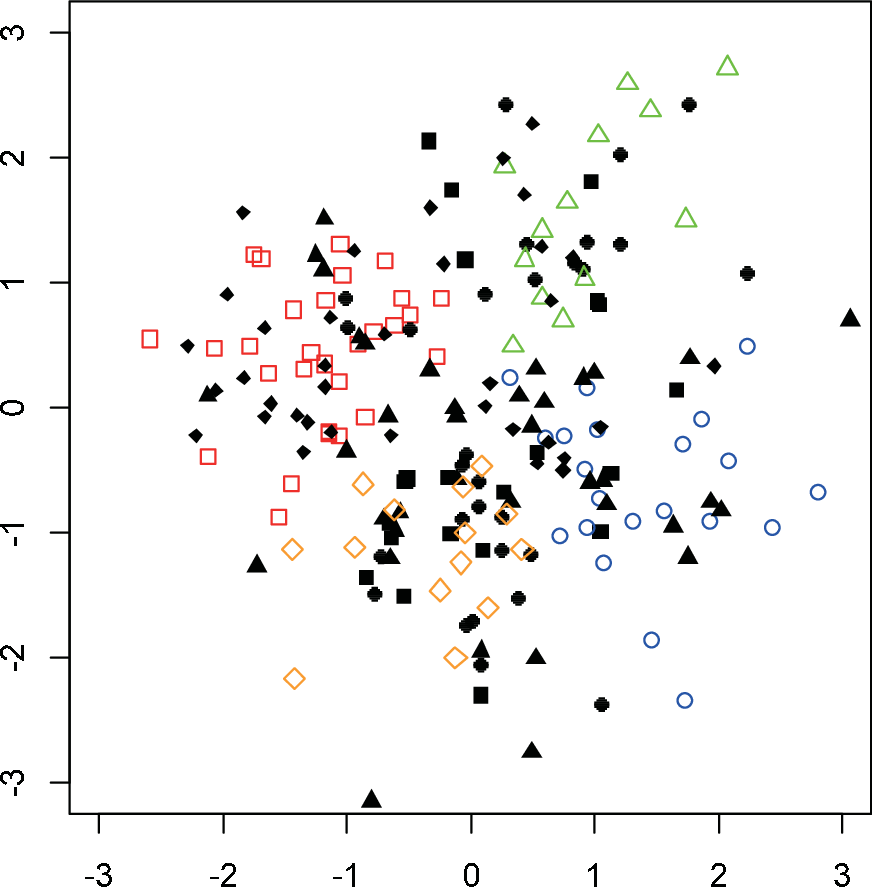}}
\caption{Plot of the result of RKM clustering for the artificial data given in Figure $\ref{data}$, 
where the black points represent misclassified objects.}
\label{RKM}
\end{figure}
The results indicate that the low-dimensional subspace revealed does not reflect the actual cluster structure 
and that the clustering result is, in fact, incorrect.

Vichi and Kiers (2001) pointed out the possibility of such problems with the RKM clustering method and 
proposed a new clustering method, 
called factorial $k$-means (FKM) clustering.
For the given data points $\bm{x}_1,\;\dots,\;\bm{x}_n$ in $\mathbb{R}^{p}$,
the number of clusters $k$, and the number of dimensions of subspace $q$, 
FKM clustering is defined by the minimization of the following loss function:
$$
FKM_n(F,\;A\mid k,\;q):=\frac{1}{n}\sum_{i=1}^{n}\min_{1\le j \le k}\|A^{T}\bm{x}_i - \bm{f}_j\|^2, 
$$
where $F:=\{\bm{f}_1,\;\dots,\;\bm{f}_{k}\},\;\bm{f}_j \in \mathbb{R}^{q}$ and $A$ is a $p\times q$ column-wise orthonormal matrix. 
When the given data points $\bm{x}_1,\;\dots,\;\bm{x}_n$ are independently drawn from a population distribution $P$,
we can rewrite the FKM objective function as 
$$
FKM(F,\;A,\;P_n):= \int\min_{\bm{f}\in F}\|A^{T}\bm{x}-\bm{f}\|^2P_n(d\bm{x}),
$$
where $P_n$ is the empirical measure of the data points $\bm{x}_1,\;\dots,\;\bm{x}_n$ in $\mathbb{R}^{p}$.
For each set of cluster centers $F$ and each $p\times q$ orthonormal matrix $A$, 
we obtain 
$$
\lim_{n\rightarrow \infty}FKM(F,\;A,\;P_n)= FKM(F,\;A,\;P):=\int\min_{\bm{f}\in F}\|A^T\bm{x}-\bm{f}\|P(d\bm{x})\quad\mathrm{a.s.}
$$
by the strong law of large numbers (SLLN).
Thus, besides $k$-means clustering and RKM clustering,
the global minimizer of $FKM(\cdot,\;\cdot,\;P_n)$ is also expected to converge almost surely to the global ones of $FKM(\cdot,\;\cdot,\;P)$, say the population global minimizers.

In this paper, 
we derive sufficient conditions for the existence of population global minimizers 
and then prove the strong consistency of FKM clustering under some regular conditions.
The framework of the proof in this paper is based on 
ones of the proof of the strong consistency of $k$-means clustering (Pollard (1981, 1982)) 
and RKM clustering (Terada (2012)).
In Pollard (1981), the proof of strong consistency of $k$-means clustering takes an inductive form.
On the other hand, the proof of  strong consistency of FKM clustering does not take such form and 
prove the consistency of FKM under the milde condition, as with Terada (2012).
In the proof of main theorem, 
first we also show that the optimal sample centres eventually lie in some compact regions on $\mathbb{R}^p$ as with Pollard (1981) and Terada (2012) and then prove the conclusion of the theorem in the same manner of the last part of 
the proof of the consistency theorem in Terada (2012).
For an arbitrary $p\times q$ column-wise orthonormal matrix $A\;(A^TA=I_q,\; q<p)$, 
an arbitrary $p$-dimensional point $\bm{x}\in \mathbb{R}^p$ and an arbitrary $q$-dimensional point $\bm{y}\in \mathbb{R}^q$,
the key inequality in this paper is that $\|A^T\bm{x}\| \le \|\bm{x}\|$
while the key equation in the strong consistency of RKM clustering (Terada (2012)) is that  $\|A\bm{y}\| = \|\bm{y}\|$.


The rest of the paper is organized as follows.
In Section $\ref{section:2}$, 
we describe the clustering algorithm of FKM to get the local minimum and the relationship between RKM clustering and FKM clustering.
We introduce prerequisites and notation in Section $\ref{section:3}$.
In Section $\ref{section:4}$, 
we prove the uniform SLLN and the continuity of the objective function of FKM clustering.
The sufficient condition for the existence of the population global minimizers and 
the strong consistency theorem of FKM clustering are stated in Section $\ref{section:5}$. 
In Section $\ref{section:6}$, 
we provide the main proof of the theorem.

%
\section{Factorial $K$-means clustering}\label{section:2}%
%
 
We will denote the number of objects and that of variables by $n$ and $p$.
Let $X=(x_{ij})_{n\times p}$ be a data matrix and $\bm{x}_{i}\;(i=1,\;\dots,\;n)$ be row vectors of $X$.
For given number of cluster $k$ and given number of dimensions of subspace $q$,
the objective function of FKM clustering is defined by 
$$
FKM_n(A,\;F,\;U\mid k,\;q) := \|XA-UF\|_{F}^2=\sum_{i=1}^{n}\min_{1\le j\le k}\|A^{T}\bm{x}_i-\bm{f}_j\|^2,
$$
where $\|\cdot\|_F$ denotes the Frobenius norm, 
$U=(u_{ij})_{n\times k}$ is a binary membership matrix, $A$ is a $p\times q$ column-wise orthonormal loading matrix, 
$F=(f_{ij})_{k\times q}$ is a centroid matrix, and $\bm{f}_j\;(j=1,\;\dots,\;k)$ are row vectors of $F$ representing
the $j$th cluster center.
$FKM_n$ can be minimized by the following alternating least-squares algorithm:
\begin{description}
\item[Step $0$.] First, initial values are chosen for $A,\;F,$ and $U$.
\item[Step $1$.] 
For each $i=1,\;\dots,\;n$ and each $j=1,\;\dots,\;k$, 
we update $u_{ij}$ by
\begin{align*}
u_{ij} = 
\begin{cases}
1 & \text{iff $\|A^T\bm{x}_i-\bm{f}_j\|^2 < \|A^T\bm{x}_i-\bm{f}_{j^{\prime}}\|^2$ for each $j^{\prime}\neq j$},\\
0 & \text{otherwise}.
\end{cases}
\end{align*}
\item[Step $2$.] 
$A$ is updated by the first $q$ eigenvectors of $X^{T}\left[U(U^{T}U)^{-1}U^{T}-I_n\right]X$, where $I_n$ is the $n$-dimensional identity matrix.
\item[Step $3$.] $F$ is updated using $(U^{T}U)^{-1}U^{T}XA$.
\item[Step $4$.] 
Finally, the value of the function $FKM_n$ for the present values of $A,\;F$, and $U$ is computed.
If the function value has decreased, 
the values of $A,\;F$, and $U$ are updated in accordance with Steps $1$-$3$.
Otherwise, the algorithm has converged.
\end{description}
This algorithm monotonically decreases the FKM objective function and 
the solution of this algorithm will be at least a local minimum point.
Thus, it is better to use many random starts to obtain the global minimum points.

Let $\hat{A},\;\hat{F}$, and $\hat{U}$ denote the optimal parameters of FKM clustering.
We can visualize the low-dimensional subspace that reflects the cluster structure by $X\hat{A}$.  
Figure $\ref{FKM}$ represents such  a visualization of the optimal subspace that results from FKM clustering for the artificial data given in Figure $\ref{data}$.
\begin{figure}
\centerline{\includegraphics[scale=0.35]{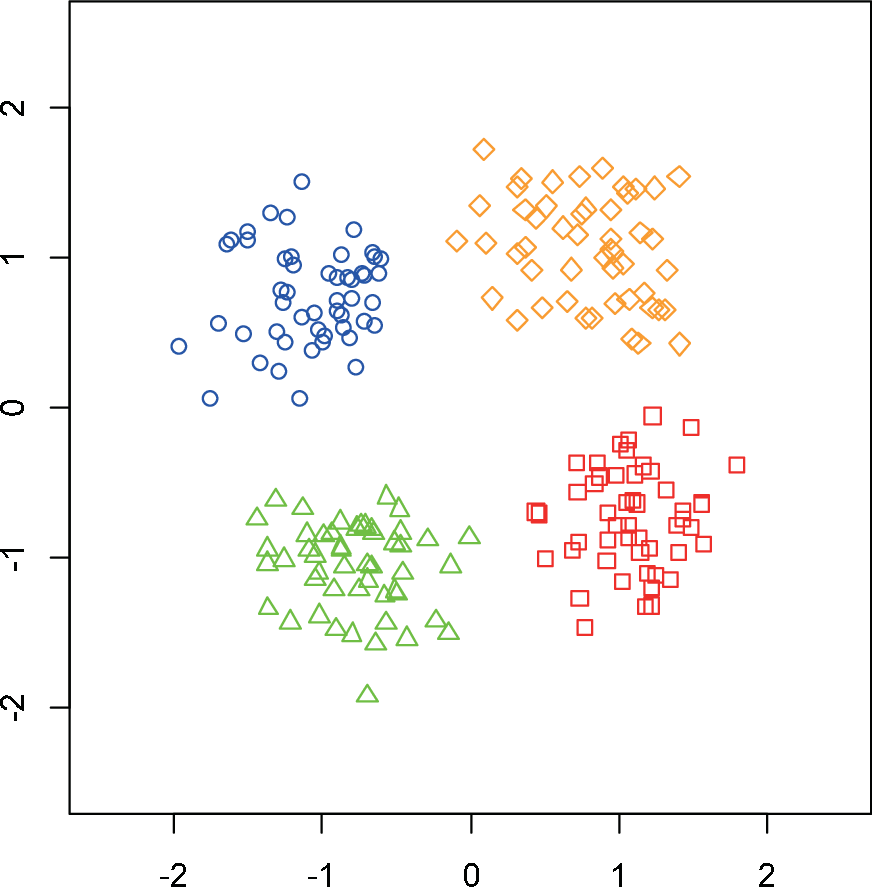}}
\caption{Plot of the result of FKM clustering for the artificial data given in Figure $\ref{data}$.}
\label{FKM}
\end{figure}

Next, we briefly discuss the relationship between the RKM clustering and FKM clustering.
The objective function of RKM clustering is defined by 
\begin{align*}
RKM_{n}(A,\;F,\;U):= \|X-UFA^T\|_{F}^2=\sum_{i=1}^{n}\min_{1\le j\le k}\|\bm{x}_i-A\bm{f}_j\|^2.
\end{align*}
This objective function can be decomposed into two terms:
\begin{align}\label{eq:2.2.1}
RKM_{n}(A,\;F,\;U)= \|X-XAA^T\|_F+\|XA-UF\|_{F}^2.
\end{align}
The first term of equation $(\ref{eq:2.2.1})$ is the objective function of the PCA procedure, 
and the second term is that of FKM clustering.
Thus, 
FKM clustering reveals the low-dimensional subspace reflecting the cluster structure more clearly
than the subspace of RKM clustering in some cases.
For more details about the relationship between RKM and FKM clustering, 
see Timmerman et al. (2010).

%
\section{Preliminaries}\label{section:3}%
%
In this paper, 
the similar notations as ones used in Pollard (1981) and Terada (2012) are used.
Let $(\Omega,\;\mathcal{F},\;P)$ be a probability space,
and $\bm{X}_1,\;\dots,\;\bm{X}_n$ be i.i.d. $p$-dimensional random variables drawn from a distribution $P$.   
Let $P_n$ denote the empirical measure based on $\bm{X}_1,\;\dots,\;\bm{X}_n$.
The set of all $p\times q$ column-wise orthonormal matrices will be denoted by $\mathcal{O}(p\times q)$.
$B_q(r)$ denotes the $q$-dimensional closed ball of radius $r$ centered at the origin.
We will define $\mathcal{R}_{k}:=\{R\subset \mathbb{R}^{q}\mid \#(R)\le k\}$, 
where $\#(E)$ is the cardinality of $E$.
We will denote  the parameter space by $\Xi_k:=\mathcal{R}_k\times \mathcal{O}(p\times q)$.
For each $M>0$,
$\mathcal{R}_{k}^{\ast}(M):=\{E\subset \mathbb{R}^{q}\mid \#(E)\le k \text{ and }E \subset B_{q}(M)\}$
and 
$\Theta_k^{\ast}(M):=\mathcal{R}_k^{\ast}(M)\times \mathcal{O}(p\times q)$.
Let $\psi:\mathbb{R}\rightarrow \mathbb{R}$ denote a non-negative decreasing function.
For each subset $F\subset \mathbb{R}^{q}$ and each $A\in\mathcal{O}(p\times q)$, 
the FKM clustering loss function with a probability measure $Q$ on $\mathbb{R}^{p}$ 
is defined by 
$$
\Psi(F,\;A,\;Q):=\int \min_{\bm{f}\in F}\psi(\|A^T\bm{x}-\bm{f}\|)Q(d\bm{x}).
$$
Write 
$$
m_k(Q):=\inf_{(F,\;A)\in \Xi_k}\Psi(F,\;A,\;Q)
$$
and
$$
m_k^\ast(Q\mid M):=\inf_{(F,\;A)\in \Theta_k^\ast(M)}\Psi(F,\;A,\;Q).
$$
For $\theta = (F,\;A) \in \Xi_k$, we will use both descriptions $\Psi(\theta,\;Q)$ and $\Psi(F,\;A,\;Q)$.
The set of population global optimizers and that of sample global optimizers will be denoted 
by $\Theta^{\prime}:=\{\theta\in \Xi_k \mid m_k(P)=\Psi(\theta,\;P)\}$ 
and 
$\Theta_n^{\prime}:=\{\theta\in \Xi_k \mid m_k(P_n)=\Psi(\theta,\;P_n)\}$, respectively.
For each $M>0$, 
let 
$\Theta^\ast:=\{\theta \in \Theta_k^\ast(M) \mid m_k^\ast(P\mid M) = \Psi(\theta,\;P)\}$
and 
$\Theta_n^\ast:=\{\theta \in \Theta_k^\ast(M) \mid m_k^\ast(P_n\mid M) = \Psi(\theta,\;P_n)\}$.
When we emphasize that $\Theta^{\prime}$ and $\Theta_n^{\prime}$ are dependent on the index $k$, 
we write $\Theta^{\prime}(k)$ and $\Theta_n^{\prime}(k)$ instead of $\Theta^{\prime}$ and $\Theta_n^{\prime}$, respectively.
One of the measurable estimators in $\Theta_n^{\prime}$ will be denoted 
by $\hat{\theta}_n$ or $\hat{\theta}_{n}(k)$.
Similarity, 
let $\hat{\theta}_n^\ast$ (or $\hat{\theta}_{n}^\ast(k)$) denote one of the measurable estimators in $\Theta_n^{\ast}$.
Existence of measurable estimators is  guaranteed by the measurable selection theorem; 
see Section $6.7$ of Pfanzagl (1994) for a detailed explanation.

Let $d_{F}(\cdot,\;\cdot)$ be the distance between two matrices based on the Frobenius norm 
and $d_{H}(\cdot,\;\cdot)$ be the Hausdorff distance, 
which is defined for finite subsets $A,\;B\subset \mathbb{R}^{q}$ as
$$
d_H(A,\;B):= \max_{\bm{a}\in A}\left\{ \min_{\bm{b}\in B}\|\bm{a}-\bm{b}\| \right\}.
$$
We will denote a product distance with $d_F$ and $d_H$ by $d$ (e.g. $d:=\sqrt{d_F^2+d_H^2}$).
As was done by Terada (2012) 
the distance between $\hat{\theta}_n$ and $\Theta^{\prime}$ is defined as 
$$
d(\hat{\theta}_n,\;\Theta^{\prime}):=\inf\{d(\hat{\theta}_n,\;\theta)\mid \theta\in \Theta^{\prime}\}.
$$
Like in Pollard (1981), 
we assume that $\psi$ is continuous and $\psi(0)=0$.
In addition, for controlling the growth of $\psi$, 
we assume that there exists $\lambda>0$ such that $\psi(2r)\le \lambda\psi(r)$ for all $r>0$.
Note that
\begin{align*}
\int \psi(\|A^T\bm{x}-\bm{f}\|)P(d\bm{x}) 
&\le
\int \psi(\|A^T\bm{x}\|+\|\bm{f}\|)P(d\bm{x}) \\
&\le
\int \psi(\|\bm{x}\|+\|\bm{f}\|)P(d\bm{x}) \\
&\le
\int_{\|\bm{f}\|>\|\bm{x}\|} \psi(2\|\bm{f}\|)P(d\bm{x}) + \int_{\|\bm{f}\|\le \|\bm{x}\|} \psi(2\|\bm{x}\|)P(d\bm{x}) \\
&\le
\psi(2\|\bm{f}\|) + \lambda \int \psi(\|\bm{x}\|)P(d\bm{x})
\end{align*}
for all $\bm{f}\in F$ and all $A\in \mathcal{O}(p\times q)$. 
Thus, $\Psi(F,\;A,\;P)$ is finite for each $F\in \mathcal{R}_k$ and $A\in \mathcal{O}(p\times q)$
as long as $\int \psi(\|\bm{x}\|)P(d\bm{x})<\infty$.

Let $R$ be a $q\times q$ orthonormal matrix, i.e., $R^{T}R=RR^{T}=I_q$.
For each $\bm{f}\in \mathbb{R}^q$ and each $A\in \mathcal{O}(p\times q)$, 
we have $AR^T\in \mathcal{O}(p\times q)$ and 
$$
\int\psi(\|A^{T}\bm{x}-\bm{f}\|)P(d\bm{x})=\int\psi(\|RA^{T}\bm{x}-R\bm{f}\|)P(d\bm{x}).
$$
Hence, 
$\Theta^{\prime}$ is not a singleton when $\Theta^{\prime}\neq \emptyset$; that is, 
FKM clustering has rotational indeterminacy, as well as RKM clustering.

%
\section{The uniform SLLN and the continuity of $\Psi(\cdot,\;\cdot,\;P)$}\label{section:4}
%
\begin{lemma}\label{Lemma:1}
Let $M$ be an arbitrary positive number.
Let $\mathcal{G}$ be the class of all $P$-integrable functions on $\mathbb{R}^p$ of the form
$
g_{(F,\;A)}(\bm{x}):=\min_{\bm{f}\in F}\psi(\|A^T\bm{x}-\bm{f}\|),
$
where $(F,\;A)$ takes all values over $\Theta_k^\ast(M)$.
Suppose that $\int\psi(\|\bm{x}\|)P(d\bm{x})<\infty$.
Then, 
\begin{align*}
\lim_{n\rightarrow \infty}\sup_{g\in \mathcal{G}}\left| \int g(\bm{x})P_n(d\bm{x}) -\int g(\bm{x})P(d\bm{x}) \right|=0\quad \mathrm{a.s.}
\end{align*}
\end{lemma}
\begin{proof}
Dehardt (1971) provided a sufficient condition for the uniform SLLN.
Thus, it is sufficient to prove that for all $\epsilon>0$, there exists a finite class of functions $\mathcal{G}_\epsilon$ such that,
for each $g\in \mathcal{G}$, there are $\dot{g}$ and $\bar{g}$ in $\mathcal{G}_{\epsilon}$ with $\dot{g}\le g \le \bar{g}$ and 
$
\int\bar{g}(\bm{x})P(d\bm{x}) -\int \dot{g}(\bm{x})P(\bm{x}) <\epsilon.
$

Choose an arbitrary $\epsilon >0$.
Let $S_{p\times q}(\sqrt{q}):=\{X\in \mathbb{R}^{p\times q}\mid \|X\|_F=\sqrt{q}\}$.
We will denote by $D_{\delta_1}$ the finite set on $\mathbb{R}^{q}$ satisfying the condition that, ,for all $\bm{f}\in B_q(M)$, there exists $\bm{g}\in D_{\delta_1}$ such that
$\|\bm{f}-\bm{g}\|<\delta_1$.
Similarly,
we will denote by $\mathcal{A}_{p\times q,\;\delta_2}$ the finite set on $S_{p\times q}(\sqrt{q})$ 
satisfying the condition that, for all $A\in S_{p\times q}(\sqrt{q})$,
there exists $B\in \mathcal{A}_{p\times q,\;\delta_2}$ such that $\|A-B\|_F<\delta_2$.
Let $\mathcal{R}_{k,\;\delta_1}:=\{F\in \mathcal{R}_{k}^{\ast}(M)\mid F \subset D_{\delta_1}\}$.
Take $\mathcal{G}_\epsilon$ as the finite class of functions of the form
$$
\min_{\bm{f}\in F_\ast}\psi(\|A_{\ast}^T\bm{x}-\bm{f}\|+\delta_1+\delta_2\|\bm{x}\|)
\quad \text{or}\quad 
\min_{\bm{f}\in F_\ast}\psi(\|A_{\ast}^T\bm{x}-\bm{f}\|-\delta_1-\delta_2\|\bm{x}\|),
$$
where $(F_\ast,\;A_\ast)$ takes all values over $\mathcal{R}_{k,\;\delta_1}\times \mathcal{A}_{p\times q,\;\delta_2}$
and 
$\psi(r)$ is defined as zero for all negative $r<0$.

For any $F=\{\bm{f}_1,\;\dots,\;\bm{f}_k\}\in \mathcal{R}_{k}^{\ast}(M)$,
there exists $F_{\ast}=\{\bm{f}_1^{\ast},\;\dots,\;\bm{f}_{k}^{\ast}\}\in \mathcal{R}_{k,\;\delta_1}$ with $\|\bm{f}_i-\bm{f}_i^{\ast}\|<\delta_1$
for each $i$.
In addition,  since $\mathcal{O}(p\times q)\subset \cup_{A_\ast\in \mathcal{A}_{p\times q,\;\delta_2}}\{A\mid \|A-A_\ast\|_F<\delta_2\}$, for any $A\in \mathcal{O}(p\times q)$, there exists $A_\ast \in \mathcal{A}_{p\times q,\;\delta_2}$ with $\|A-A_\ast\|_F<\delta_2$.
Corresponding to each $g_{(F,\;A)}\in \mathcal{G}$, 
choose
$$
\bar{g}_{(F,\;A)}(\bm{x}):=\min_{\bm{f}\in F_\ast}\psi(\|A_\ast^T\bm{x}-\bm{f}\|+\delta_1+\delta_2\|\bm{x}\|)
$$
and 
$$
\dot{g}_{(F,\;A)}(\bm{x}):=\min_{\bm{f}\in F_\ast}\psi(\|A_\ast^T\bm{x}-\bm{f}\|-\delta_1-\delta_2\|\bm{x}\|).
$$
Since $\psi$ is a monotone function and 
$$
\|A_\ast^{T}\bm{x}-\bm{f}_j^{\ast}\|-\delta_1-\delta_2\|\bm{x}\|
\le
\|A^{T}\bm{x}-\bm{f}_j\|
\le
\|A_\ast^{T}\bm{x}-\bm{f}_j^{\ast}\|+\delta_1+\delta_2\|\bm{x}\|
$$
for each $i$ and each $\bm{x}\in \mathbb{R}^{p}$,
we have 
$
\dot{g}_{(F,\;A)} \le g_{(F,\;A)} \le \bar{g}_{(F,\;A)}.
$

Choosing $R>0$ to be greater than $(M+\delta_1)/\sqrt{q}$ (or $(M + \delta_1)/(\sqrt{q} + \delta_2)$),
we obtain
\begin{align*}
&\int \left[\bar{g}_{(F,\;A)}(\bm{x})-\dot{g}_{(F,\;A)}(\bm{x})\right]P(d\bm{x})\\
\le
&\int \sum_{i =1}^{k}
\bigl[
\psi(\|A_{\ast}^{T}\bm{x}-\bm{f}_{i}^{\ast}\|+\delta_1+\delta_2\|\bm{x}\|)
-\psi(\|A_{\ast}^{T}\bm{x}-\bm{f}_{i}^{\ast}\|-\delta_1-\delta_2\|\bm{x}\|)
\bigr]P(d\bm{x})\\
\le
&k\sup_{\|\bm{x}\|\le R}\sup_{\bm{f}\in B_{q}(M)}\sup_{A\in S_{p\times q}(\sqrt{q})}
\bigl[
\psi(\|A^{T}\bm{x}-\bm{f}\|+\delta_1+\delta_2\|\bm{x}\|)\\
&\qquad\qquad
-\psi(\|A^{T}\bm{x}-\bm{f}\|-\delta_1-\delta_2\|\bm{x}\|)
\bigr]
+2k\lambda^{m}\int_{\|\bm{x}\|\ge R}\psi(\|\bm{x}\|)P(d\bm{x}),
\end{align*}
where $m\in \mathbb{N}$ is chosen to satisfy the requirement that $\sqrt{q}+\delta_2\le 2^{m-1}$.
The second term in the last bound of the inequality directly above can be less than
$\epsilon/2$ by choosing $R$ to be sufficiently large.
Note that $\psi$ is uniform continuous on a bounded set.
The first term can be less than $\epsilon/2$ by choosing $\delta_1,\;\delta_2>0$ to be sufficiently small.
Therefore, 
the sufficient condition of the uniform SLLN for $\mathcal{G}$ is satisfied, and the proof is complete.
\qed
\end{proof}

\begin{lemma}\label{Lemma:2}
Let $M$ be an arbitrary positive number.
Suppose that $\int\psi(\|\bm{x}\|)P(d\bm{x})<\infty$.
Then, $\Psi(\cdot,\;P)$ is continuous on $\Theta_k^\ast(M)$.
\end{lemma}
\begin{proof}
This lemma can be proven in a similar manner as the proof of Lemma $\ref{Lemma:1}$.
If $(F,\;A),\;(G,\;B)\in \Theta_k^\ast(M)$ is chosen to satisfy $d_{H}(F,\;G)<\delta_1$ and $\|A-B\|_F<\delta_2$, 
then for each $\bm{g}\in G$ there exists $\bm{f}(\bm{g})\in F$ such that $\|\bm{g}-\bm{f}(\bm{g})\|<\delta_1$.
Choosing $R$ to be larger than $M+\delta_1$,
we obtain 
\begin{align}\label{lemma2:eq1}
&\Psi(F,\;A,\;P)-\Psi(G,\;B,\;P)\nonumber\\
=
&\int\left[\min_{\bm{f}\in F}\psi(\|A^{T}\bm{x}-\bm{f}\|)-\min_{\bm{g}\in G}\psi(\|B^{T}\bm{x}-\bm{g}\|)\right]P(d\bm{x})\nonumber\\
\le
&\int\max_{\bm{g}\in G}
\left[\psi(\|A^{T}\bm{x}-\bm{f}(\bm{g})\|)-\psi(\|B^{T}\bm{x}-\bm{g}\|)\right]P(d\bm{x})\nonumber\\
\le
&\int\sum_{\bm{g}\in G}\left[ \psi(\|B^T\bm{x}-\bm{g}\|+\delta_1+\delta_2\|\bm{x}\|)-\psi(\|B^{T}\bm{x}-\bm{g}\|) \right]P(d\bm{x})\nonumber\\
\le
&k\sup_{\|\bm{x}\|\le R}\max_{\bm{g}\in G}\left[ \psi(\|B^T\bm{x}-\bm{g}\|+\delta_1+\delta_2\|\bm{x}\|)-\psi(\|B^{T}\bm{x}-\bm{g}\|) \right]\nonumber\\
&+2\sum_{\bm{g}\in G} \int_{\|\bm{x}\|\ge R}\psi(\|B^T\bm{x}-\bm{g}\|+\delta_1+\delta_2\|\bm{x}\|)P(d\bm{x})\nonumber\\
\le
&k\sup_{\|\bm{x}\|\le R}\max_{\bm{g}\in G}\left[ \psi(\|B^T\bm{x}-\bm{g}\|+\delta_1+\delta_2\|\bm{x}\|)-\psi(\|B^{T}\bm{x}-\bm{g}\|) \right]\nonumber\\
&+2k\lambda^{m}\int_{\|\bm{x}\|\ge R}\psi(\|\bm{x}\|)P(d\bm{x}),
\end{align}
where $m\in \mathbb{N}$ is chosen to satisfy the condition that $2+\delta_2\le 2^m$.
By choosing $R$ to be sufficiently large and $\delta_1,\;\delta_2>0$ to be sufficiently small, 
the last bound in the inequality $(\ref{lemma2:eq1})$ can be less than $\epsilon$.
Since for each $\bm{f}\in F$ there exists $\bm{g}(\bm{f})\in G$ such that $\|\bm{g}-\bm{g}(\bm{f})\|<\delta_1$, 
the other inequality needed for continuity is obtained by interchanging $(F,\;A)$ and $(G,\;B)$ in the inequality $(\ref{lemma2:eq1})$.
\qed
\end{proof}

%
\section{Consistency theorem}\label{section:5}%
%
\subsection{Existence of population global optimizers}

Our purpose is to prove that
$\lim_{n\rightarrow \infty}d(\hat{\theta}_n,\;\Theta^{\prime})=0\;\mathrm{a.s.}$ 
under some regularity conditions.
However, there is a possibility that $\Theta^{\prime}$ is empty.
Therefore, first, 
we provide sufficient conditions for the existence of population global optimizers.
\begin{prop}\label{Prop:1}
Suppose that $\int\psi(\|\bm{x}\|)P(d\bm{x})<\infty$ and that $m_j(P)>m_k(P)$ for $j=1,\;2,\;\dots,\;k-1$.
Then, $\Theta^{\prime}\neq \emptyset$.
Furthermore, 
there exists $M>0$ such that $F\subset B_{q}(5M)$ for all $(F,\;A)\in \Theta^{\prime}$.
\end{prop}
\begin{proof}
See Appendix \ref{app}.
\qed
\end{proof}

Under the assumption of Proposition $\ref{Prop:1}$, 
we can prove that $\Psi(\cdot,\;P)$ ensures the identification condition, which is a requirement of
the consistency theorem.
\begin{col}\label{Col:1}
Suppose that $\int\psi(\|\bm{x}\|)P(d\bm{x})<\infty$ and that $m_j(P)>m_k(P)$ for $j=1,\;2,\;\dots,\;k-1$.
Then, there exists $M_0>0$ such that for each $M>M_0$
$$
\inf_{\theta\in \Theta_{\epsilon}^\ast(M)}\Psi(\theta,\;P) > \inf_{\theta\in \Theta^{\prime}}\Psi(\theta,\;P)\quad \text{for all }\epsilon >0.
$$
where $\Theta_{\epsilon}^\ast(M):=\{\theta \in \Theta_k^\ast(M) \mid d(\theta,\;\Theta^{\prime})\ge \epsilon\}$.
\end{col}
\begin{proof}
See Appendix \ref{app}.
\qed
\end{proof}

\subsection{Strong consistency of FKM clustering}

If the parameter space is restricted to $\Theta_k^\ast(M) \subset \Xi_k$, 
we easily obtain the strong consistency of FKM clustering.
Since $\Theta_k^\ast(M)$ is compact, 
we have $\Theta^\ast\neq \emptyset$ and the identification condition:
$$
\inf_{\theta \in \Theta_\epsilon^\ast(M)}\Psi(\theta,\;P) > 
\inf_{\theta \in \Theta^\ast}\Psi(\theta,\;P)
\quad \text{for all } \epsilon>0
$$
where $\Theta_\epsilon^\ast(M):=\{\theta \in \Theta_k^\ast(M) \mid d(\theta,\;\Theta^\ast) \ge \epsilon\}$.

\begin{prop}\label{Prop:2}
Let M be an arbitrary positive number. Suppose that $\int\psi(\|\bm{x}\|)P(d\bm{x})<\infty$.
Then, 
$$
\lim_{n\rightarrow \infty}d(\hat{\theta}_n^\ast,\;\Theta^{\ast})=0\;
\mathrm{a.s.}
,\;\text{and }
\lim_{n\rightarrow \infty} m_k^\ast(P_n\mid M)=m_k^\ast(P\mid M)\;
\mathrm{a.s.}
$$
\end{prop}
\begin{proof}
From Lemma $\ref{Lemma:1}$ and Lemma $\ref{Lemma:2}$, we already obtain the uniform SLLN and the continuity of $\Psi(\cdot,\;P)$ on $\Theta_k^\ast(M)$.
Thus, the proof of this proposition is given by the similar argument of the last part of the proof of the consistency theorem.
\qed
\end{proof}
This fact is very important in the proof of Lemma $\ref{Lemma:4}$.
Using this fact, the proof of the main theorem does not necessary take  an inductive form with the number of cluster $k$ and we can prove the consistency under the mild condition.  

We cannot assume the uniqueness condition since FKM clustering has rotational indeterminacy.
In this study, as Terada (2012) did previously, 
we assume that $m_j(P)>m_k(P)$ for $j=1,\;\dots,\;k-1$.
This condition implies that an optimal set $F(k)$ of cluster centres has $k$ distinct elements.
When we do not use the fact in Proposition $\ref{Prop:2}$, 
we may need more strict condition $m_1(P)>m_2(P)>\dots>m_k(P)$ and 
the proof of the main theorem takes an inductive form with the number of cluster $k$ as with Pollard (1981).
The following theorem provides sufficient conditions for the strong consistency of FKM clustering.

\begin{theorem}\label{Theorem:1}
Suppose that $\int\psi(\|\bm{x}\|)P(d\bm{x})<\infty$ and that $m_j(P)>m_k(P)$ for $j=1,\;\dots,\;k-1$.
Then, $\Theta^{\prime}\neq \emptyset$,
$$
\lim_{n\rightarrow \infty}d(\hat{\theta}_n,\;\Theta^{\prime})=0\;
\mathrm{a.s.}
,\;\text{and }
\lim_{n\rightarrow \infty} m_k(P_n)=m_k(P)\;
\mathrm{a.s.}
$$
\end{theorem}
\begin{proof}
See Section $\ref{section:5}$.
\end{proof}

Note that
if there exists a specific $A$ such that $\Psi(A,\;F,\;P)=0$ for all $F$; that is, 
the population distribution, $P$, is degenerate and the number of dimensions with the support of $P$ is given as $p-q$,
$m_j(P)>m_k(P)$ for $j=1,\;\dots,\;k-1$ is not satisfied.
%
\section{Proof of the theorem}\label{section:6}%
%

Since the theorem deals with almost sure convergence,
there might exist null subsets of $\Omega$ on which the strong consistency does not hold.
Therefore, throughout the proof, $\Omega_1$ denotes the set obtained by avoiding a possible null set from $\Omega$.

First, we prove that there exists $M>0$ such that,
for sufficiently large $n$, at least one center of the estimator $F_n\in \mathcal{R}_k$ is contained in $B_{q}(M)$.

\begin{lemma}\label{Lemma:3}
Suppose that $\int\psi(\|\bm{x}\|)P(d\bm{x})<\infty$.
Then,
there exists $M>0$ such that
\begin{align*}
P\left(\bigcup_{n=1}^{\infty}\bigcap_{m=n}^{\infty}\{\omega \mid \forall (F_m,\;A_m)\in \Theta_m^{\prime};\;F_m(\omega)\cap B_q(M)\neq \emptyset\}  \right)=1.
\end{align*}
\end{lemma}
\begin{proof}
Choose an $r>0$ to satisfy the condition that $P(B_p(r))>0$.
Let us take $M$ to be sufficiently large to ensure that $M>r$ and 
\begin{align}\label{lemma3:eq1}
\psi(M-r)P(B_p(r))>\int \psi(\|\bm{x}\|)P(d\bm{x}).
\end{align}
Note that $m_k(P_n)\le \Psi(F,\;A,\;P_n)$ for all $F\in\mathcal{R}_k$ and all $A\in \mathcal{O}(p\times q)$.
Let $F_0$ be the singleton that consists of only the origin.
By the SLLN, we obtain
$$
\Psi(F_0,\;A,\;P_n) = \int \psi(\|A^{T}\bm{x}\|)P_n(d\bm{x}) \rightarrow \int \psi(\|A^T\bm{x}\|)P(d\bm{x}) \quad \mathrm{a.s.}
$$
for all $A\in \mathcal{O}(p\times q)$.
Since $\|A^T\bm{x}\|\le \|\bm{x}\|$, 
we have 
$$
\int \psi(\|A^T\bm{x}\|)P(d\bm{x}) \le \int\psi(\|\bm{x}\|)P(d\bm{x})
$$
for all $A\in \mathcal{O}(p\times q)$.

Let $\Omega^{\prime}:=\{\omega\in \Omega_1 \mid \forall n\in \mathbb{N};\;\exists m\ge n;\;F_m(\omega) \cap B_q(M)\}$.
For all $\omega\in \Omega^{\prime}$, there exists a subsequence $\{n_l\}_{l\in \mathbb{N}}$ such that $F_{n_l}(\omega) \cap B_q(M)=\emptyset$.
Since $\|A^{T}\bm{x}-\bm{f}\|\le \|\bm{f}\|-\|\bm{x}\|>M-r$ for all $\bm{x}\in B_p(r)$, all $\bm{f}\in B_{q}(M)$, and all $A\in \mathcal{O}(p\times q)$, 
we have
\begin{align*}
\lim\sup_l \Psi(F_{n_l},\;A_{n_l},\;P_{n_l}) 
&\ge 
\lim\sup_{l} \frac{1}{n_l}\sum_{i \in \{i\mid \bm{X}_i\in K\}}\min_{\bm{f}\in F_{n_l}}\psi(\|A_{n_l}^{T}\bm{X}_{i}-\bm{f}\|)\\
&\ge
\lim\sup_l \frac{1}{n_l}\sum_{i \in \{i\mid \bm{X}_i\in K\}}\psi(M-r)\\
&\ge
\psi(M-r)P(B_p(r)).
\end{align*}
From the assumptions made on the values of $M$, 
we have 
$$
\lim\sup_l \Psi(F_{n_l},\;A_{n_l},\;P_{n_l}) > \int \psi(\|\bm{x}\|)P(d\bm{x}),
$$
which contradicts $m_k(P_n)\le \Psi(F,\;A,\;P_n)$ for all $F\in\mathcal{R}_k$ and all $A\in \mathcal{O}(p\times q)$.
Therefore, we obtain $P(\Omega^{\prime})=0$; that is, 
$$
P\left( \bigcup_{n=1}^{\infty}\bigcap_{m=n}^{\infty} \{\omega \mid \forall (F_m,\;A_m)\in \Theta_m^{\prime} ;\;F_m(\omega)\cap B_q(M)\neq \emptyset\}\right)=1.
$$
\qed
\end{proof}

By Lemma $\ref{Lemma:3}$, 
without loss of generality, 
we can assume that 
each $F_n$ contains at least one element of $B_q(M)$ when $n$ is sufficiently large.
The next lemma indicates that 
there exists $M>0$ such that $B_q(5M)$ contains all the estimators of centers when $n$ is sufficiently large.


\begin{lemma}\label{Lemma:4}
Under the assumption of the theorem, 
there exists $M>0$ such that 
\begin{align*}
P\left(\bigcup_{n=1}^{\infty}\bigcap_{m=n}^{\infty}\{\omega \mid \forall (F_m,\;A_m)\in \Theta_m^{\prime};\;F_m(\omega)\subset B_q(5M)\}  \right)=1.
\end{align*}
\end{lemma}
\begin{proof}
Choose $\epsilon>0$ sufficiently small such that $\epsilon +m_k(P)<m_{k-1}(P)$.
Let us take $M>0$ to satisfy the inequality $(\ref{lemma3:eq1})$ and 
\begin{align}\label{lemma4:eq1}
\lambda \int_{\|\bm{x}\|\ge 2M}\psi(\|\bm{x}\|)P(d\bm{x})<\epsilon.
\end{align}

Suppose that $F_n$ contains at least one center outside $B_q(5M)$.
By Lemma $\ref{Lemma:3}$, 
when $n$ is sufficiently large, 
$F_n$ must contain at least one center in $B_q(M)$, say $\bm{f}_1\in B_q(M)$.
Since $\{\bm{x}\mid \|A^T\bm{x}\|\ge 2M\}\subset \{\bm{x}\mid \|\bm{x}\|\ge 2M\}$, we have
\begin{align*}
\int_{\|A^T\bm{x}\|\ge 2M} \psi(\|A^T\bm{x}-\bm{f}_1\|)P_n(d\bm{x}) 
&\le
\int_{\|\bm{x}\|\ge 2M} \psi(\|A^{T}\bm{x}-\bm{f}_1\|)P_n(d\bm{x})\\
&\le
\int_{\|\bm{x}\|\ge 2M} \psi(\|\bm{x}\|+\|\bm{f}_1\|)P_n(d\bm{x})\\
&\le
\lambda\int_{\|\bm{x}\|\ge 2M} \psi(\|\bm{x}\|)P_n(d\bm{x})
\end{align*}
for all $A\in \mathcal{O}(p\times q)$.
Let $F_n^{\ast}$ denote the set obtained by deleting all centers lying outside $B_q(5M)$ from $F_n$.
Since $(F_{n}^{\ast},\;A)\in \Theta_{k-1}^\ast(5M)$ for all $A\in \mathcal{O}(p\times q)$,
we have 
$$
\Psi(F_{n}^{\ast},\;A,\;P_n)\ge m_{k-1}^\ast(P_n\mid 5M) \ge m_{k-1}(P_n)
$$
for all $A\in \mathcal{O}(p\times q)$.
For each $\bm{x}\in B_p(2M)$ and each $A\in \mathcal{O}(p\times q)$, 
we have 
$$
\|A^T\bm{x}-\bm{f}\|\ge \|\bm{f}\|-\|\bm{x}\|>3M \quad\text{for all }\bm{f}\notin B_q(5M)
$$
and 
$$
\|A^T\bm{x}-\bm{g}\|\le \|\bm{x}\|+\|\bm{g}\|<3M \quad\text{for all }\bm{g}\in B_q(5M).
$$
Thus, we obtain
$$
\int_{\|x\|<2M}\min_{\bm{f}\in F_n}\psi(\|A^T\bm{x}-\bm{f}\|)P_n(d\bm{x})=\int_{\|x\|<2M}\min_{\bm{f}\in F_n^{\ast}}\psi(\|A^T\bm{x}-\bm{f}\|)P_n(d\bm{x})
$$
for all $A\in \mathcal{O}(p\times q)$.

Let $\Omega^{\ast}:=\{\omega\in \Omega_1 \mid \forall n \in \mathbb{N};\;\exists m\ge n;\;\exists (F_m,\;A_m)\in \Theta_m^{\prime};\;F_m(\omega)\not\subset B_q(5M)\}$.
By the axiom of choice, 
for an arbitrary $\omega\in \Omega^{\ast}$, there exists a subsequence $\{n_l\}_{l\in \mathbb{N}}$ such that
$F_m(\omega)\not\subset B_q(5M)$.
By Proposition $\ref{Prop:2}$, 
we have 
$$
\lim_{n\rightarrow \infty} m_{k-1}^\ast(P_n\mid 5M)=m_{k-1}^\ast(P\mid 5M)\quad \mathrm{a.s.}
$$
For any $(F,\;A)\in \Xi_{k}$,
we have 
\begin{align}\label{lemma4:eq2}
m_{k-1}(P)
&\le
m_{k-1}^\ast(P\mid 5M)
\le 
\lim\inf_l\Psi(F_{n_{l}}^{\ast},\;A_n,\;P_n)
\le
\lim\sup_l\Psi(F_{n_{l}}^{\ast},\;A_{n_l},\;P_{n_l})\nonumber\\
&\le
\lim\sup_n \Biggl[ \int_{\|\bm{x}\|<2M}\min_{\bm{f}\in F_{n}}\psi(\|A_n^{T}\bm{x}-\bm{f}\|)P_{n}(d\bm{x})\nonumber\\
&\qquad\qquad\quad+\int_{\|\bm{x}\|\ge2M}\psi(\|A_n^{T}\bm{x}-\bm{f}_1\|)P_{n}(d\bm{x}) \Biggr]\nonumber\\
&\le 
\lim\sup_n \left[ \Psi(F_n,\;A_n,\;P_n) +\lambda \int_{\|\bm{x}\|\ge 2M}\psi(\|\bm{x}\|)P_n(d\bm{x})\right]\nonumber\\
&\le 
\lim\sup_n\Psi(F,\;A,\;P_n) +\lambda \int_{\|\bm{x}\|\ge 2M}\psi(\|\bm{x}\|)P_n(d\bm{x}).
\end{align}
Choose $(\bar{F},\;\bar{A})\in \Theta^{\prime}$ as $(F,\;A)\in\Xi_{k}$ in the last bound of the above inequality.
By the assumption of $M>0$ and the SLLN, for a sufficiently large $n$,
the last bound of the inequality $(\ref{lemma4:eq2})$ can be less than
$m_{k}(P)+\epsilon$, 
which is a contradiction.
Therefore, 
we obtain
\begin{align*}
P\left(\bigcup_{n=1}^{\infty}\bigcap_{m=n}^{\infty}\{\omega \mid \forall (F_m,\;A_m)\in \Theta_m^{\prime};\;F_m(\omega)\subset B_q(5M)\}  \right)=1.
\end{align*}
\qed
\end{proof}

Hereafter, 
$M$ denotes a positive value satisfying inequalities $(\ref{lemma3:eq1})$ and $(\ref{lemma4:eq1})$.
According to Lemma $\ref{Lemma:4}$, 
for all $(F_n,\;A_n)\in \Theta_n^\prime$, $F_n\in \mathcal{R}_k^\ast(5M)$
when $n$ is sufficiently large.
Since $\mathcal{R}_k^\ast(5M)$ is compact, $\Theta_k^\ast(5M)$ is also compact.

By the uniform SLLN, the continuity of $\Psi(\cdot,\;\cdot,\;P)$ on $\Theta_k^\ast(5M)$ and Lemma $\ref{Lemma:4}$, 
the conclusion of the theorem for the cluster number $k$ can be proved in the same manner as 
was done for the last part of the proof of the consistency theorem in Terada (2012).

Choose $\theta_\ast\in \Theta_k^\ast(5M)$ such that $d(\theta_\ast,\;\Theta^{\prime})>0$.  
Write 
\begin{align*}
\tilde{\theta}_n
=
\begin{cases}
\hat{\theta}_n	& \text{if }\hat{\theta}_n\in \Theta_k^\ast(5M) \\
\theta_\ast     & \text{if }\hat{\theta}_n\notin \Theta_k^\ast(5M)
\end{cases}.
\end{align*}
By Lemma $\ref{Lemma:4}$, we have $\tilde{\theta}_n=\hat{\theta}_n$ for a sufficiently large $n$.
Since $\Psi(\hat{\theta}_n,\;P_n) = \inf_{\theta \in \Xi_k} \Psi(\theta,\;P_n)$, 
we have
$$
\lim\sup_{n}\left[ \Psi(\tilde{\theta}_n,\;P_n)-\inf_{\theta\in \Theta^{\prime}}\Psi(\theta,\;P_n)\right]\le 0\quad \mathrm{a.s.}
$$
Since $\lim\sup_n\psi(\theta_0,\;P_n)=m_k(P)$ for any $\theta_0\in \Theta^{\prime}$, 
$$
\lim\sup_n \inf_{\theta\in \Theta^{\prime}}\Psi(\theta,\;P_n) \le \lim\sup_n\Psi(\theta_0,\;P_n)=m_k(P)\quad \mathrm{a.s.}
$$
Hence, 
we have
\begin{align*}
0
&\ge 
\lim\sup_n \Psi(\tilde{\theta}_n,\;P_n)-\lim\sup_n \inf_{\theta\in \Theta^{\prime}} \Psi(\theta,\;P_n)\nonumber \\
&\ge
\lim\sup_n \Psi(\tilde{\theta}_n,\;P_n)-m_k(P)\quad \mathrm{a.s.}
\end{align*}
Let $\Theta_{\epsilon}^\ast(5M):=\{\theta\in \Theta_k^\ast(5M)\mid d(\theta,\;\Theta^{\prime})\ge\epsilon\}$.
By the uniform SLLN applied to $\Theta_k^\ast(5M)$,
we obtain
\begin{align*}
\lim\inf_n \inf_{\theta\in \Theta_{\epsilon}^\ast(5M)}\Psi(\theta,\;P_n)\ge 
\inf_{\theta\in \Theta_{\epsilon}^\ast(5M)}\Psi(\theta,\;P)\quad \mathrm{a.s.}
\end{align*}
for all $\epsilon>0$.
Fix an arbitrary $\epsilon>0$.
By Corollary $\ref{Col:1}$, 
$$
\lim\inf_n \inf_{\theta\in \Theta_{\epsilon}^\ast(5M)}\Psi(\theta,\;P_n)
>
\lim\sup_n \Psi(\tilde{\theta}_n,\;P_n)\quad \mathrm{a.s.}
$$
Thus,
for any $\omega\in \Omega_1$ there exists $n_0\in \mathbb{N}$ such that
$$
\inf_{\theta\in \Theta_{\epsilon}^\ast(5M)}\Psi(\theta,\;P_n)
>
\Psi(\tilde{\theta}_n,\;P_n)
$$
for all $n\ge n_0$.
Conversely, 
suppose that $d(\tilde{\theta}_n,\;\Theta^{\prime})\ge \epsilon$ for some $n\ge n_0$.
Then, we have
$$
\inf_{\theta\in \Theta_{\epsilon}^\ast(5M)}\Psi(\theta,\;P_n)=\Psi(\tilde{\theta}_n,\;P_n),
$$
which is a contradiction.
Thus,
we obtain 
$$
\lim_{n\rightarrow \infty}d(\tilde{\theta}_n,\;\Theta^{\prime})=0\quad \mathrm{a.s.}
$$
By $\tilde{\theta}_n=\hat{\theta}_n$ for a sufficiently large $n$, 
it follows that
$$
\lim_{n\rightarrow \infty}d(\hat{\theta}_n,\;\Theta^{\prime})=0\quad \mathrm{a.s.}
$$
Moreover, 
by the continuity of $\Psi(\cdot,\;P)$ on $\Theta_k^\ast(5M)$, we obtain
$$
\lim_{n\rightarrow \infty}m_k(P_n)=m_k(P)\quad \mathrm{a.s.}
$$

%
%
\section{Conclusion}\label{section:7}%
%
In this study, 
we proved the strong consistency of FKM clustering under i.i.d. sampling 
by using the frameworks of the proof for the consistency of $k$-means clustering (Pollard (1981)) 
and the consistency of RKM clustering (Terada (2012)).
The compactness of parameter space is not a requirement for the sufficient condition of the strong consistency for FKM clustering, as well as $k$-means clustering and RKM clustering.
As with the $k$-means and RKM clustering, 
the proof is based on Blum-DeHardt uniform SLLN (Peskir (2000)).
Thus, for the consistency of FKM clustering, 
stationarity and ergodicity is only required and the i.i.d. condition is also not necessary.
We also derived the sufficient condition for ensuring the existence of population global optimizers of FKM clustering. 
Moreover, we proved the uniform SLLN and continuity of the FKM objective function
in the proof of the consistency theorem.

In the future, 
we will derive the rate of convergence of FKM clustering estimators.
%




\appendix

%
\section{Existence of $\Theta^{\prime}$}\label{app}%
%
Here we prove the existence of population global optimizers.
\begin{lemma}\label{Lemma:A1}
Suppose that $\int\psi(\|\bm{x}\|)P(d\bm{x})<\infty$.
There exists $M>0$ such that
$$
\inf_{A\in \mathcal{O}(p\times q)}\Psi(F^{\prime},\;A,\;P) > \inf_{\theta\in \Theta_k^\ast(M)}\Psi(\theta,\;P)
$$
for all $F^{\prime}\in\mathcal{R}_k$ satisfying $F^{\prime}\cap B_q(M)=\emptyset$.
\end{lemma}
\begin{proof}
Conversely, suppose that, for all $M>0$, 
there exists $F^{\prime}\in\mathcal{R}_k$ such that $F^{\prime}\cap B_q(M)=\emptyset$ and 
\begin{align*}
\inf_{A\in \mathcal{O}(p\times q)}\Psi(F^{\prime},\;A,\;P) \le \inf_{\theta\in \Theta_k^\ast(M)}\Psi(\theta,\;P).
\end{align*}
Choose $r>0$ to satisfy that 
the ball $B_p(r)$ has a positive $P$ measure; 
that is $P(B_p(r))>0$.
Let $M$ be sufficiently large such that   $M>r$ and 
that it satisfies inequality $(\ref{lemma3:eq1})$.
Since $\|A^T\bm{x}-\bm{f}\|\ge \|\bm{f}\| -\|A^{T}\bm{x}\|>M-r$ for all $\bm{f}\notin B_{q}(M)$ and all $\bm{x}\in B_p(r)$, 
we have
\begin{align*}
\int \psi(\|\bm{x}\|)P(\bm{x}) 
&\ge \inf_{\theta\in \Theta_k^\ast(M)}\Psi(\theta,\;P)
\ge \inf_{A\in \mathcal{O}(p\times q)}\Psi(F^{\prime},\;A,\;P)\\
&\ge \inf_{A\in \mathcal{O}(p\times q)}\int_{\bm{x}\in B_p(r)}\min_{\bm{f}\in F^{\prime}}\psi(\|A^T\bm{x}-\bm{f}\|)P(d\bm{x})\\
&\ge \phi(M-r)P(B_p(r)).
\end{align*}
This is a contradiction.
\qed
\end{proof}

\begin{lemma}\label{Lemma:A2}
Suppose that $\int\psi(\|\bm{x}\|)P(d\bm{x})<\infty$, and for $j=2,\;3,\;\dots,\;k-1$, $m_j(P) > m_{k}(P)$.
There exists $M>0$ such that, for all $F^{\prime}\in\mathcal{R}_k$ satisfying $F^{\prime}\not\subset B_q(5M)$,
$$
\inf_{A\in \mathcal{O}(p\times q)}\Psi(F^{\prime},\;A,\;P) > 
\inf_{\theta\in \Theta_k^\ast(5M)}\Psi(\theta,\;P).
$$
\end{lemma}
\begin{proof}
Choose $M>0$ to be sufficiently large to satisfy inequalities $(\ref{lemma3:eq1})$ and $(\ref{lemma4:eq1})$.
Suppose that, for all $M>0$, there exists $F^{\prime}\in\mathcal{R}_k$ satisfying $F^{\prime}\not\subset B_q(5M)$ and
$$
\inf_{A\in \mathcal{O}(p\times q)}\Psi(F^{\prime},\;A,\;P) \le 
\inf_{\theta\in \Theta_k^\ast(5M)}\Psi(\theta,\;P).
$$
Let $\mathcal{R}_k^{\prime}$ be the set of such $F^{\prime}$ and then
$$
m_{k}(P)=\inf_{\theta\in \mathcal{R}_k^{\prime}\times \mathcal{O}(p\times q)}\Psi(\theta,\;P).
$$
According to Lemma \ref{Lemma:A1}, 
each $F^{\prime}\in \mathcal{R}_k^{\prime}$ includes at least one point on $B_{q}(M)$, say $\bm{f}_1$.
For all $\bm{x}$ satisfying $\|\bm{x}\|<2M$ and all $A\in\mathcal{O}(p\times q)$, 
we obtain
$$
\|A^T\bm{x}-\bm{f}\|>3M \quad \text{for all $\bm{f}\not\in B_q(5M)$}
$$
and 
$$
\|A^T\bm{x}-\bm{g}\|<3M \quad \text{for all $\bm{g}\in B_q(M)$}.
$$
Thus, 
$$
\int_{\|\bm{x}\| < 2M} \min_{\bm{f}\in F^{\prime}}\psi(\|A^T\bm{x}-\bm{f}\|)P(d\bm{x})
=
\int_{\|\bm{x}\| < 2M} \min_{\bm{f}\in F^{\ast}}\psi(\|A^T\bm{x}-\bm{f}\|)P(d\bm{x}),
$$
where the set $F^{\ast}$ is obtained by deleting all points outside $B_q(5M)$ from $F^{\prime}$.
Since
$
\int_{\|\bm{x}\| \ge 2M}\psi(\|A^T\bm{x}-\bm{f}_1\|)P(d\bm{x}) 
\le
\lambda \int_{\|\bm{x}\| \ge 2M}\psi(\|\bm{x}\|)P(d\bm{x})
$, we obtain that
\begin{align*}
&\Psi(F_{k}^{\prime},\;A,\;P) + \lambda \int_{\|\bm{x}\| \ge 2M}\psi(\|\bm{x}\|)P(d\bm{x})\\
&\ge
\int_{\|\bm{x}\| < 2M} \min_{\bm{f}\in F^{\ast}}\psi(\|A^T\bm{x}-\bm{f}\|)P(d\bm{x})
+
\int_{\|\bm{x}\| \ge 2M}\psi(\|A^T\bm{x}-\bm{f}_1\|)P(d\bm{x})\\
&\ge
\Psi(F^{\ast},\;A,\;P) \ge m_{k-1}(P)
\end{align*}
for all $A\in \mathcal{O}(p\times q)$.
It follows that
$m_k(P)+\epsilon \le m_{k-1}(P)$, 
which is a contradiction.
\qed
\end{proof}
Let us consider $M>0$ to be sufficiently large to satisfy inequalities $(\ref{lemma3:eq1})$ and $(\ref{lemma4:eq1})$.
Write $\Theta_k:=\mathcal{R}_k^\ast(5M)\times \mathcal{O}(p\times q)$.
Proposition $\ref{Prop:1}$ and Corollary $\ref{Col:1}$ can be proved in the same way as Proposition $1$ and Corollary $1$ in Terada (2012).
\begin{proof}[Proof of Proposition $\ref{Prop:1}$]%
According to Lemma \ref{Lemma:A2},
$$
\inf_{\theta\in \Xi_k}\Psi(\theta,\;P)=\inf_{\theta\in \Theta_k}\Psi(\theta,\;P).
$$
Moreover, 
for any $\theta\in (\mathcal{R}_k\setminus \mathcal{R}_k^\ast(5M))\times \mathcal{O}(p\times q)$,
$m_k(P)<\Psi(\theta,\;P)$.
Thus, we only have to prove $\Theta^{\prime}\neq \emptyset$.

Let $C:=\{\Psi(\theta,\;P)\mid \theta \in \Theta_k\}$ and then $m_k(P)=\inf C$.
By the definition of the infimum, for all $x > m_k(P)$, there exists $c\in C$ such that $c<x$.
By the axiom of choice, 
we can obtain a sequence $\{c_n\}_{n\in \mathbb{N}}$ such that 
$c_n \rightarrow m_k(P)$ as $n\rightarrow \infty$.
Using the axiom of choice again, 
we can obtain a sequence $\{\theta_{n}\}_{n\in \mathbb{N}}$ such that $\Psi(\theta_n,\;P)\rightarrow m_k(P)$ as $n\rightarrow \infty$.

By the compactness of $\Theta_k$, there exists a convergent subsequence of $\{\theta_n\}_{n\in \mathbb{N}}$, 
say $\{\theta_{n_i}\}_{i\in \mathbb{N}}$.
Let $\theta_{\ast}\in \Theta_k$ denote the limit of subsequence $\{\theta_{n_i}\}_{i\in \mathbb{N}}$, 
i.e., $\theta_{m_i}\rightarrow \theta_\ast$ as $i\rightarrow \infty$.
Since $\Psi(\cdot,\;P)$ is continuous on $\Theta_{k}$, $\Psi(\theta_\ast,\;P)=m_k(P)$.
Hence, we obtain $\Theta^{\prime}\neq \emptyset$.
\qed
\end{proof}

\begin{proof}[Proof of Corollary $\ref{Col:1}$]%
Let $\Theta_\epsilon:=\{\theta_k\in \Theta_k \mid \Psi(\theta_k,\;P)=m_k(P)\}$.
Conversely, 
suppose that there exists $\epsilon >0$ such that $\inf_{\theta\in \Theta_{\epsilon}}\Psi(\theta,\;P)=\inf_{\theta\in \Theta^{\prime}}\Psi(\theta,\;P)$.
By the definition of the infimum, 
there exists a sequence $\{\theta_n\}_{n\in \mathbb{N}}$ on $\Theta_{\epsilon}$ such that
$\Psi(\theta_n,\;P)\rightarrow m_k(P)$ as $n\rightarrow \infty$.
By compactness of $\Theta_k$, there exists a convergent subsequence of $\{\theta_n\}_{n\in \mathbb{N}}$, 
say $\{\theta_{m_i}\}_{i \in \mathbb{N}}$.
Let $\theta_{\ast}\in \Theta_k$ denote the limit of subsequence $\{\theta_{m_i}\}_{i\in \mathbb{N}}$.
Since $\theta_{m_i}\rightarrow \theta_\ast$ as $i\rightarrow \infty$, 
we have $d(\theta_{m_i},\;\theta_\ast)<\epsilon$ for a sufficiently large $i$, 
which is a contradiction.
\qed
\end{proof}

\end{document}